\documentclass{amsart}

 \usepackage{amssymb}

\newtheorem*{theorem*}{Theorem}
\newtheorem{theorem}{Theorem}[section]
\newtheorem{lemma}[theorem]{Lemma}

\theoremstyle{remark}

\theoremstyle{definition}

\newtheorem{example}[theorem]{Example}

\numberwithin{equation}{section}
\numberwithin{theorem}{section}

\newcommand{\bm}{\mathfrak{m}}

\newcommand{\tld}{\widetilde }

\newcommand{\bfX}{\pmb{X}}
\newcommand{\X}{\overline{X}}
\newcommand{\E}{\mathcal{E}}

\DeclareMathOperator{\ini}{{in}}
\DeclareMathOperator{\rank}{{rank}}

\newcommand{\eg}{{\itshape e.g.} }
\newcommand{\ie}{{\itshape i.e.} }

\usepackage{enumerate}

\begin{document}

 \title[On matrix Schubert varieties]{On the $F$-rationality and cohomological properties of matrix Schubert varieties}
\author{Jen-Chieh Hsiao}

\address{Department of Mathematics\\
  National Cheng Kung University\\
  No.1 University Road, Tainan 70101, Taiwan}
\email{jhsiao@mail.ncku.edu.tw}
\thanks{The author was partially supported by NSF under grant DMS~0901123.}

\maketitle

\begin{abstract}
We characterize complete intersection matrix Schubert varieties, generalizing the classical result on one-sided ladder determinantal varieties. 
We also give a new proof of the $F$-rationality of matrix Schubert varieties.
Although it is known that such varieties are $F$-regular (hence $F$-rational) by the global $F$-regularity of Schubert varieties, our proof is of independent interest since it does not require the Bott--Samelson resolution of Schubert varieties. As a consequence, this provides an alternative proof of the classical fact that Schubert varieties in flag varieties are normal and have rational singularities.

\end{abstract}

\section{Introduction}

Matrix Schubert varieties (MSVs) were introduced by W. Fulton in his theory of degeneracy loci of maps of flagged vector bundles \cite{Ful92}.   Such varieties are reduced and irreducible. 
Classical (one-sided ladder) determinantal varieties are special examples of MSVs (they are so-called vexillary MSVs). 
Just like one-sided ladder determinantal varieties \cite{GL00}, \cite{GM00}, MSVs can be identified (up to product of an affine space) as the opposite big cells of the corresponding Schubert varieties. This observation in \cite{Ful92} implies that the MSVs are normal and Cohen--Macaulay, since Schubert varieties are (see \cite{Ra85}).

The Cohen--Macaulay property of MSVs was re-established by A. Knutson and E. Miller \cite{KM05} using the Gr{\"o}bner
basis theory, pipe dreams, and their theory of subword complexes.
Interestingly, this gives a new proof of the Cohen--Macaulayness of Schubert varieties by the following principle. 
\begin{theorem}[\cite{KM05} Theorem~2.4.3]\label{principle}
Let $\mathcal{C}$ be a local condition that holds for a variety $X$ whenever it holds for the product of $X$ with any vector space. Then $\mathcal{C}$ holds for every Schubert variety in every flag variety if and only if $\mathcal{C}$ holds for all MSVs.
\end{theorem}

\subsection{$F$-rationality of MSVs}
In the same spirit, the first part of this paper is devoted to a new proof of $F$-rationality of MSVs. $F$-rationality is a notion that arises from the theory of tight closure introduced by M. Hochster and C. Huneke \cite{Hu96} in positive characteristic. The results of \cite{Smi97} and \cite{Hara98} establish a connection between $F$-rationality and the notion of rational singularity in characteristic $0$: A normal variety in characteristic $0$ has at most rational singularities if and only if it is of $F$-rational type. Therefore, by Theorem~\ref{principle} the $F$-rationality of matrix Schubert varieties is equivalent to the classical fact that Schubert varieties are normal and have at most rational singularities (see \eg \cite{Bri05} and \cite{BK05} for the classical proofs of the later statement using the Bott--Samelson resolution). 

Two other notions in tight closure theory will also be used later: $F$-regularity and $F$-injectivity. The relation between these properties is
\[ \text{regular} \implies F \text{-regular} \implies F\text{-rational} \implies F\text{-injective}.\]
We remark that MSVs are in fact $F$-regular by Theorem~\ref{principle} and the global $F$-regularity of Schubert varieties \cite{LRT06} (again, this relies on the Bott--Samelson resolution).

Our proof of $F$-rationality of MSVs utilizes the results of Schubert determinantal ideals in \cite{KM05} as well as the techniques developed in \cite{CH97}, where A. Conca and J. Herzog prove that arbitrary (possibly two-sided) ladder determinant varieties are $F$-rational. However, it is still unknown whether such varieties are $F$-regular.

One of the key ingredients in our proof is the following
\begin{theorem}[\cite{CH97} Theorem~1.2] \label{ch97tool}
Let $R$ be a complete local Cohen--Macaulay ring and $c$ be a non zero-divisor of $R$ such that $R[{1 \over c}]$ is $F$-rational and $R/cR$ is $F$-injective. Then $R$ is $F$-rational. 
\end{theorem}
After recalling several known facts in the theory of tight closure (section \ref{frational}), we will see that the most essential step is to find $c$ such that $R_w[{1 \over c}]$ is $F$-rational and that the initial ideal $\ini_{<}(\langle c \rangle + I_w)$ of $\langle c \rangle +I_w$ is Cohen--Macaulay (where $<$ is any antidiagonal term order, $R_w$ and $I_w$ is the coordinate ring and the defining ideal of the MSV associated to the partial permutation $w$ as defined in section \ref{msv}). This goal is achieved by choosing 
 $c = x_{i_0,w(i_o)}$ where $i_0$ is the smallest number such that
$\{ (p,q) \mid p>i_0, q > w(i_0) \} \cap \E_{>0}(w) \neq \emptyset.$ See section \ref{msv} for unexplained notation.

\subsection{Complete intersection MSVs}

Since MSVs are Cohen--Macaulay, 
it is then natural to ask when such varieties are smooth, complete intersection, or Gorenstein. Classically, characterizations of Gorenstein ladder determinantal varieties are obtained in \cite{Co95}, \cite{Co96}, and \cite{GM00}. In one-sided cases, the characterization can be generalized as the following.
Recall that there exists a characterization of smooth (respectively, Gorenstein) Schubert varieties \cite{LS90} (respectively, \cite{WY06}). Since the singular (respectively, non-Gorenstein) locus of a Schubert variety is closed and invariant under the Borel subgroup action, the opposite big cell must be contained in the singular (respectively, non-Gorenstein) locus. Hence, a Schubert variety is smooth (respectively, Gorenstein) if and only if its corresponding MSV is so. Therefore, one can deduce a criterion of smooth (respectively, Gorenstein) MSVs by the corresponding result for Schubert varieties. See \cite[section 3.5]{WY06}  for more details.

The second goal of this paper is to characterize complete intersection MSVs.   We explain the characterization as the following. See sections \ref{msv} and \ref{ci} for unexplained notation and more details.

\begin{theorem*}[Theorem \ref{thmCI}]  The matrix Schubert variety $\X_w$ associated to a permutation $w$ is a complete intersection if and only if, for any $(p,q)$ in the diagram of $w$ with $ r_{p,q}(w)>0$, \ie $(p,q) \in D_{>0}(w)$, $w_{(p,q)}$ is a permutation matrix in $GL_{r_{p,q}(w)}$ such that $\X_{w_{(p,q)}}$ is a complete intersection. Here, $w_{(p,q)}$ is 
the connected (solid) square submatrix of size
$r_{p,q}(w)$ whose southeast corner lies at $(p-1,q-1)$.
In fact, in this case
\[\{ x_{p,q} \mid (p,q) \in D_{=0}(w) \} \bigsqcup \{ \det \bfX_{\overline{(p,q)}} \mid (p,q) \in D_{>0}(w) \}
\] is a set of generators of $I_w$ with cardinality $|D(w)|$, the codimension of $\X_w$ in $M_{n \times n}$, where
$\bfX_{\overline{(p,q)}}$ is 
the connected (solid) square submatrix of size
$r_{p,q}(w)+1$ whose southeast corner lies at $(p,q)$.
\end{theorem*}
Theorem~\ref{thmCI} generalizes a result in \cite{GS95} for one-sided ladder determinantal varieties. The proof of Theorem~\ref{thmCI} uses Nakayama's lemma and the properties of Schubert determinantal ideals established in \cite{KM05}.

After this work is finished, A. Woo and H. Ulfarsson give a criterion of locally complete intersection Schubert varieties. Theorem \ref{thmCI} may be recovered by their criterion (see \cite[Corollary 6.3]{UW11} and the comment after that).
\subsection{} This paper is organized as follows. We will recall some preliminary facts about matrix Schubert varieties as well as tight closure theory in section \ref{msv} and \ref{frational}, respectively. 
The proof of $F$-rationality of MSVs is in section \ref{msvarefrational}.
Section \ref{ci} is devoted to the characterization of complete intersection MSVs.

\subsection{Acknowledgements}
The author would like to thank A. Knutson, E. Miller, K. Smith, U. Walther, and A. Woo for their comments and suggestions about this work.
Special thanks go to the referee for carefully reading this paper and many useful suggestions on the presentation of the results

\section{Matrix Schubert varieties} \label{msv}

We recall some fundamental facts about matrix Schubert varieties (see \cite{Ful92},  \cite{KM05}, and \cite{MS} for more information).

Denote $M_{l \times m}$ the space of $l \times m$ matrices over a field $K$.
An $l \times m$ matrix $w \in M_{l \times m}$ is called a {\bf partial permutation} if all entries of $w$ are equal to $0$ except for  at most one entry equal to $1$ in each row and column. If $l=m$ and $w \in GL_l$, then $w$ is called a {\bf permutation}. An element $w$ in the permutation group $S_n$ will be identified as a permutation matrix (also denoted by $w$) in $GL_n$ via
\begin{equation*}
w_{i,j}= \begin{cases} 1 & \text{ if } w(i)=j, \\ 0 & \text{ otherwise.}
\end{cases}
\end{equation*}
Let $ K[\bfX]$ be the coordinate ring of $M_{l \times m}$ where $\bfX = (x_{i,j})$ is the generic $l \times m$ matrix of variables. 
For a matrix $Z \in M_{l \times m}$, denote $Z_{[p,q]}$ the upper-left $p \times q$ submatrix of $Z$. Similarly, $\bfX_{[p,q]}$ denotes the upper-left $p \times q$ submatrix of $\bfX$. The rank of $Z_{[p,q]}$ will be denoted by $\rank(Z_{[p,q]}) := r_{p,q}(Z)$.

Given a partial permutation $w \in M_{l \times m}$, the { \bf matrix Schubert variety} $\X_w$ is the subvariety 
\[ \X_w = \{ Z \in M_{l \times m} \mid r_{p,q}(Z) \leq r_{p,q}(w) \text{ for all } p,q \} 
\]  in $M_{l \times m}$.
The classical (one-sided ladder) determinantal varieties are special examples of MSVs.

It is known that MSVs are reduced and irreducible. Denote $$R_w=K \left[\X_w \right] = K[\bfX]/I_w$$ the coordinate ring of $\X_w$. The defining ideal $I_w$ of $\X_w$ (called {\bf Schubert determinantal ideal}) is generated by all minors  of size $r_{p,q}(w)+1$ in $\bfX_{[p,q]}$. One can reduce the generating set of $I_w$ as the following. Consider the {\bf diagram} of $w$
\[D(w) =\{(i,j) \in [1,l]\times[1,m] : w(i)>j \text{ and } w^{-1}(j)>i \},\]
\ie $D(w)$ consists of elements that are neither due east nor due south of a nonzero entry of $w$. The {\bf essential set} of $w$ is defined to be
\[ \mathcal{E}(w) = \{ (p,q) \in D(w) : (p+1,q) \notin D(w) \text{ and } (p,q+1) \notin D(w) \}.
\]
One can check that (see \cite[Lemma 3.10]{Ful92})
\begin{equation}\label{generator}
\begin{aligned} 
I_w &= \langle \text{ minors of size } r_{p,q}(w)+1 \text{ in } \bfX_{[p,q]} : (p,q) \in D(w)\rangle \\
&= \langle \text{ minors of size } r_{p,q}(w)+1 \text{ in } \bfX_{[p,q]} : (p,q) \in \E(w)\rangle
\end{aligned}
\end{equation}
Also, the codimension of $\X_w$ in $M_{l \times m}$ is the cardinality $|D(w)|$ of $D(w)$ which is actually the Coxeter length of $w$ when $w$ is a permutation.
We often need to consider certain subsets of $D(w)$ or $\E(w)$. For that, we will put the conditions as subscripts to indicate the constraints. For examples, $D_{=0}(w) = \{ (p,q) \in D(w) \mid r_{p,q}(w)=0 \}$ and  $D_{>0}(w) = \{ (p,q) \in D(w) \mid r_{p,q}(w)>0 \}$.

Questions on $\X_w$ for a partial permutation $w \in M_{l \times m}$ is often reduced to the cases where $w$ is a permutation. More precisely, extend $w$ to the permutation $\tld{w} \in S_n, n=l+m$ via
\begin{equation*}
\tld{w}(i) = 
\begin{cases}
j & \text{ if } w_{i,j}=1 \text{ for some }  j  \\
\min \{ [m+1,n] \setminus \{\tld{w}(1),\dots,\tld{w}(i-1)\}\} & \text{ if } w_{i,j}=0 \text{ for all }
j \\
\min \{ [1,n] \setminus \{\tld{w}(1),\dots,\tld{w}(i-1)\}\}& \text{ if } i > l.
\end{cases}
\end{equation*}
Then $D(w)=D(\tld{w})$, $E(w)=E(\tld{w})$, and the defining ideals $I_w$ and $I_{\tld{w}}$ share the same set of generators. Therefore,
\begin{equation}\label{reducetopermutation}
\X_{\tld{w}} \cong \X_w \times K^{n^2 - lm}.
\end{equation} 

The following substantial results due to A. Knutson and E. Miller is indispensable in the proofs of our main theorems. Recall that a term order on $K[\bfX]$ is called {\bf 
antidiagonal} if  the initial term of every minor of $\bfX$ is its antidiagonal term. We will fix an antidiagonal term order $<$ and simply write $\ini(I)$, $\ini(f)$ as the initial ideal of an ideal $I$ and the leading term of an element $f$, respectively. We will call an antidiagonal term of a minor of size $r$ an {\bf antidiagonal of size $\mathbf{r}$}.
\begin{theorem}[\cite{KM05}]\label{initial}
Fix any antidiagonal term order. Then
\begin{enumerate}
\item The generators of $I_w$ in \eqref{generator} constitute a Gr{\"o}bner basis, \ie 
\[\ini (I_w) = \langle \text{antidiagonals of size }r_{p,q}(w)+1 \text{ in } \bfX_{[p,q]} :
(p,q) \in \E(w) \rangle; 
\] 
\item $\ini(I_w)$ is a Cohen--Macaulay square-free monomial ideal.
\end{enumerate}
\end{theorem}

\section{$F$-rationality and $F$-injectivity}\label{frational}
Recall that in the theory of tight closure, a Noetherian ring is {\bf $F$-rational} if all its parameter ideals are tightly closed. There is a weaker notion called  $F$-injectivity. A Noetherian ring $R$ is {\bf $F$-injective} if for any maximal ideal $\frak{m}$  of $R$ the map on the local cohomology module $H^i_{\frak{m}}(R)$ induced by the Frobenius map is injective for all $i$.
We collect some facts concerning $F$-rationality and $F$-injectivity. See \cite{Hu96}, or \cite{BH93} for convenient resources.
\begin{theorem}\label{fact1} Let $R$ be a Noetherian ring.
\begin{enumerate}
\item $R$ is $F$-rational if and only if $R_{\bm}$ is $F$-rational for any maximal ideal $\bm$.
\item If $R$ is an $F$-rational ring that is a homorphic image of a Cohen--Macaulay ring, then $R_S$ is $F$-rational for any multiplicative closed set $S$ of $R$.
\end{enumerate}
\end{theorem}
\begin{theorem}\label{fact2} Let $(R,\bm)$ be a Noetherian local ring.
\begin{enumerate}
\item $R$ is $F$-injective if and only if $\widehat{R}$ is $F$-injective.
\item Suppose in addition $R$ is excellent, then $R$ is $F$-rational if and only if $\widehat{R}$ is $F$-rational.
\end{enumerate}
\end{theorem}

\begin{theorem}\label{fact3} Let $R$ be a positive graded $K$-algebra, where $K$ is a field of positive characteristic. Let $\bm$ be the unique maximal graded ideal of $R$.
\begin{enumerate}
\item $R$ is $F$-injective if and only if $R_{\bm}$ is $F$-injective.
\item $R$ is $F$-rational if and only if $R[T]$ is $F$-rational for any variable $T$ over $R$.
\item Suppose in addition that $K$ is perfect. Then $R$ is $F$-rational if and only if $R_{\bm}$ is $F$-rational.
\end{enumerate}
\end{theorem}
\section{Matrix Schubert varties are $F$-rational} \label{msvarefrational}
Fix an antidiagonal term order. Denote $J_w = \ini(I_w)$ the initial ideal of $I_w$.
In this section, the ground field $K$ is perfect  and of positive characteristic.
As mentioned in the introduction, consider $c:=x_{i_0,w(i_0)}$ where $i_0$ 
is the smallest number such that
$\{ (p,q) \mid p>i_0, q > w(i_0) \} \cap \E_{>0}(w) \neq \emptyset.$
Note that such $i_0$ exists exactly when $\E_{>0}(w) \neq \emptyset$ (or equivalently $R_w$ is not regular).
We make this assumption (existence of $c$) for Lemmas \ref{lem1}, \ref{lem2}, \ref{lem3} and set $(p_0,q_0)=(i_0,w(i_0))$.
Note also that for this particular choice of $i_0$,
\begin{equation}\label{D=0}
\{(p,q) \ne (p_0,q_0) \mid p \le p_0,q \le q_0 \} \subseteq [\{(p,q)\mid p \le p_0 \} \cap D(w)]  \subseteq D_{=0}(w).\end{equation}
 In particular, the only nonzero entry in $w_{[p_0,q_0]}$ is $(p_0,q_0)$.

In the following, we use the notation $[p_1, \dots, p_t\mid q_1, \dots, q_t]$ to denote the size $t$ minor of the submatrix of $\bfX$ involving the rows of indices $p_1, \dots, p_t$ and the columns of indices $q_1, \dots, q_t$.
\begin{lemma}\label{lem1} Let $\Delta$ be any minor in $\bfX$ such that $c\mid \ini \Delta$. Then 
$\Delta \in \langle c \rangle + J_w$ and hence so is $\Delta - \ini \Delta $.
\end{lemma}
\begin{proof}
Write $\Delta=\left[ p_t,\dots,p_1,p_0,p_1',\dots,p_s' \mid
q_s',\dots,q_1',q_0,q_1,\dots,q_t \right]$, so \[
\ini \Delta = \left(\prod^s_{i=1}x_{p_i',q_i'}\right)\cdot c \cdot \left(\prod^t_{j=1} x_{p_j,q_j}\right),
\] where $p_s' > \dots >p'_1>p_0 > p_1>\dots>p_t$ and $q_s'<\dots<q_1'<q_0<q_1<\dots<q_t$.

 Use induction on $t$. When $t=0$, then $\Delta=\left[ p_0,p_1',\dots,p_s' \mid q_s',\dots,q_1',q_0 \right]$. If $s=0$, $\Delta=c= \ini \Delta$ is obviously in $\langle c\rangle +J_w$. Suppose $s>0$, 
Expanding $\Delta$ along the first row, we get
$$\begin{aligned}
\Delta 
=\, &(-1)^s c \left[ p_1',\dots,p_s' \mid q_s',\dots,q_1'  \right] +\\
& \sum_{i=1}^s (-1)^{i+1}x_{p_0,q_i'} [ p_1', \dots, p_s' \mid q_s', \dots, \widehat{q_i'}, \dots, q_1', q_0 ].\end{aligned}
$$
Since $0=r_{p_0,q_1'}(w)=\dots=r_{p_0,q_s'}(w)$ by \eqref{D=0},
$ \langle x_{p_0,q'_1},\dots,x_{p_0,q'_s} \rangle \subseteq J_w
$ and hence $\Delta \in  \langle c \rangle +J_w$. 

Suppose $t>0$. 
For $1 \leq i \leq t$, set \[
\Delta_i = \left[ p_{t-1},\dots,p_1,p_0,p_1',\dots,p_s' \mid
q_s',\dots,q_1',q_0,q_1,\dots,\widehat{q_i},\dots,q_t \right].
\]
Note that $c \mid \ini \Delta_i$ for $ 1 \le i \le t$, since $c$ is on the antidiagonal of $\Delta_i$ (the row deleted is above $c$ and the column deleted is to the right of $c$). So $\langle \Delta_i \mid 1 \le i \le t \rangle \subseteq \langle c\rangle +J_w$ by the inductive hypothesis. Again, expanding $\Delta$ along the first row, we see that 
\[\Delta  \in \langle \Delta_i \mid 1 \le i \le t \rangle + \langle x_{p_t,q_j'} \mid j=1,\dots ,s \rangle + \langle x_{p_t,q_0}\rangle.
\]
Once again $ [\langle x_{p_t,q_j'}\mid j=1,\dots ,s \rangle + \langle x_{p_t,q_0}\rangle ] \subseteq J_w$, since $0=r_{p_t,q_1'}(w)=\dots=r_{p_t,q_s'}(w)=r_{p_t,q_0}(w)$ by \eqref{D=0}. Therefore, $\Delta \in \langle c \rangle + J_w$ as desired.
\end{proof}
\begin{lemma}\label{lem2}
$\ini(\langle c \rangle +I_w) = \langle c \rangle + J_w.$
\end{lemma}
\begin{proof}
The containment $\ini(\langle c \rangle +I_w) \supseteq \langle c \rangle + J_w$ is obvious.
Conversely, let $cf-g \in \langle c \rangle +I_w$ for some $f \in K[\bfX]$ and $g \in I_w$. If $\ini(cf)\neq \ini(g)$, then $\ini(cf-g)= \ini(cf)$ or $\ini(-g)$. In either case, $\ini(cf-g)\in \langle c \rangle + J_w.$ 

So we may assume that 
$\ini(cf)=\ini(g)$. Write $g=m_1\Delta_1 +m_2\Delta_2+\cdots+m_u\Delta_u$ where the $m_i$'s are monomial elements in $K[\bfX]$ and the $\Delta_i$'s are minors in the generating set of $I_w$.
If $\ini(cf-g)$ is a term in $cf$, then $\ini(cf-g) \in \langle c\rangle$. Also, if $\ini(cf-g)= \ini(m_i \Delta_i)$ for some $i$, then $\ini(cf-g)=m_i\dot \ini(\Delta_i) \in J_w$.
Therefore, we may assume that $\ini(cf-g)$ is a term of $m_{i_0}\Delta_{i_0}$ for some $i_0$ and that $\ini(cf-g)$ is neither $\ini(m_{i_0}\Delta_{i_0})$ nor a term of $cf$. This implies that $\ini(m_{i_0}\Delta_{i_0})$ is a term of $cf$ and hence $c \mid \ini(m_{i_0}\Delta_{i_0})=m_{i_0} \dot \ini(\Delta_{i_0})$. If $c \mid m_{i_0}$, then $\ini(cf-g) \in \langle c \rangle$ since it is a term of $m_{i_0}\Delta_{i_0}$.
Otherwise, $c \mid \ini(\Delta_{i_0})$. Then by Lemma~\ref{lem1}, $\Delta_{i_0}- \ini(\Delta_{i_0}) \in \langle c \rangle + J_w$. Therefore, $m_{i_0}\Delta_{i_0}- \ini(m_{i_0}\Delta_{i_0}) \in \langle c \rangle +J_w$. Now, since
$\ini(cf-g)$ is a term of $m_{i_0}\Delta_{i_0}- \ini(m_{i_0}\Delta_{i_0})$ and since $\langle c \rangle +J_w$ is a monomial ideal, we conclude that $\ini(cf-g) \in \langle c \rangle +J_w$.
\end{proof}
\begin{lemma}\label{lem3}
$R_w/cR_w$ is $F$-injective.
\end{lemma}
\begin{proof} 
By Theorem 2.1 in \cite{CH97}, it suffices to show that $K[\bfX]/\ini(\langle c \rangle + I_w)$ is Cohen--Macaulay and $F$-injective. By Lemma~\ref{lem2}, $\ini(\langle c \rangle +I_w) = \langle c \rangle + J_w.$ Also, by Theorem~\ref{initial}{\it (2)} $J_w$ is a Cohen--Macaulay square-free monomial ideal. So $\langle c \rangle +J_w$ is also a square-free monomial ideal, and hence $K[\bfX]/(\langle c \rangle + J_w)$ is $F$-injective by the discussion in the paragraph before Corollary of \cite{CH97} involving Fedder's criterion. The Cohen--Macaulayness of $K[\bfX]/(\langle c \rangle + J_w)$ follows from that fact the $c$ is a non zero-divisor on $K[\bfX] / J_w$. 

To see this, first note that $c=x_{i_0,w(i_0)}=x_{p_0,q_0} \notin J_w$. Suppose for some $z \in K[\bfX]$ we have $cz \in J_w$. We will show that $z \in J_w$. By Theorem~\ref{initial}, we may assume $z$ is a monomial and $cz = r D$ for some monomial $r \in K[\bfX]$ and some antigonal $D \in J_w$.
If $c \mid r $, then $z = {r \over c}D \in J_w$. Therefore, we may assume $c \nmid r$. Then $c \mid D$. We finish the proof by showing that ${D \over c} \in J_w$.

Write \[ D= \left( \prod^s_{i=1} x_{p_i',q_i'} \right)\cdot c \cdot \left( \prod^t_{j=1} x_{p_j,q_j} \right)
\] where $p1' > p_2'>\dots > p_s'>p_0>p_1 >\dots >p_t$ and $q_1'<q_2'<\dots
<q_s'<q_0<q_1<\dots<q_t$. Since $c=x_{p_0,q_0} \notin J_w$, either $s>0$ or $t>0$. Note also that $D$ is of size $(s+t+1)$, so  $r_{p_1',q_t}(w) \leq s+t$ by Theorem~\ref{initial}{\it (1)}. On the other hand, as mentioned before  the only nonzero entry in $w_{[p,q]}$ is $(p_0,q_0)$, so
\[
\begin{aligned} &\left\{ \text{nonzero entries in }w_{[p'_1,q_t]} \right\} \\
&\supseteq \{ (p_0,q_0) \} \bigsqcup \left\{ \text{nonzero entries in }w_{[p_1,q_t]}\right\}
\bigsqcup \left\{ \text{nonzero entries in }w_{[p'_1,q_s']}\right\}.
\end{aligned}
\]
Therefore, $1+r_{p_1,q_t}(w)+r_{p'_1,q'_s}(w) \leq r_{p'_1,q_t} \leq s+t$ and hence either $r_{p_1,q_t}(w) < t$ or $r_{p'_1,q'_s}(w) <s$.
We conclude that either $\prod^s_{i=1} x_{p_i',q_i'} \in J_w$ or $\prod^t_{i=1} x_{p_i,q_i} \in J_w$, so ${D\over c} \in J_w$.
\end{proof}
\begin{theorem} \label{frational}
$R_w = K[\bfX]/I_w$ is $F$-rational.
\end{theorem}
\begin{proof}
By \eqref{reducetopermutation} and Theorem~\ref{fact3}{\it (2)}, we may assume that $w \in S_n$ is a permutation. Use induction on $n$.
If $R_w$ is regular (this includes the cases $n=1,2$), then it is $F$-rational.
Suppose $n>2$ and $R_w$ is not regular. Then the element $c= x_{p_0,q_0}$ described as above exists. 
By Lemma~\ref{lem3}, $R_w / cR_w$ is $F$-injective. Hence, $\widehat{R_w}/c\widehat{R_w}$ is $F$-injective by Theorem~\ref{fact2}{\it (1)} and \ref{fact3}{\it (1)}. So by Theorem~\ref{ch97tool}, \ref{fact2}{\it (2)}, and \ref{fact3}{\it (3)}, it suffices to show that $R_w[{1 \over c}]$ is $F$-rational.

For $(p,q) \notin \Gamma=  \{ (p,q) \mid p=p_0 \text{ or } q=q_0 \}$, consider the change of variables
\[ x_{p,q}' = x_{p,q} -c^{-1}x_{p,q_0}x_{p_0,q}. 
\] 
Set $\bfX' = \left(x'_{p,q}\mid  (p,q) \notin \Gamma \right)$. Then
\[K[ \bfX ] \left[ {1\over c}\right] = K[\bfX'][x_{p,q} \mid (p,q) \in \Gamma ] \left[ {1 \over c}\right] \]
in the field of fraction of $K[\bfX]$.\\
Let $w'$ be the permutation obtained by deleting  the $p_0$th row and the $q_0$th column of $w$ and let $I_{w'}$ be the corresponding Schubert determinantal ideal in $K[\bfX']$.
Denote $I= I_w \cdot K[\bfX]\left[ {1\over c}\right]$ the extended ideal of $I_w$ and set
\[ I' = \left[ I_{w'}+ \langle x_{p,q} \mid (p,q) \in \Gamma, p<p_0 \text{ or } q<q_0 \rangle \right] \cdot 
K[\bfX'][x_{p,q} \mid (p,q) \in \Gamma ] \left[ {1 \over c}\right] \]
We claim that 
\begin{equation}\label{I=I'}
I=I' \end{equation}
It follows from \eqref{I=I'} that
\[  {K[\bfX] \over I_w} \left[ {1 \over c} \right] = {K[\bfX'] \over I_{w'}} \left[x_{p,q}\mid (p,q) \in \Gamma ,
p\ge p_0 \text{ and }q\ge q_0\right] \left[ {1 \over c} \right] \]
By inductive hypothesis, $K[\bfX']/I_{w'}$ is $F$-rational. So Theorem~\ref{fact1}{\it (2)} and Theorem~\ref{fact3}{\it (2)} imply that $R_w\left[ {1 \over c} \right]$ is $F$-rational. Therefore, it suffices to prove \eqref{I=I'}.

We prove \eqref{I=I'} by showing that the generators of $I$ belongs to $I'$ and conversely.
First observe that
\begin{itemize}
\item[(a)] For $(p,q) \in \Gamma$ satisfying $p <p_0$ or $q<q_0$, by \eqref{D=0} $x_{p,q} \in I_w$.
\item[(b)] Fix a $(p,q)$ satisfying $p <p_0$ or $q<q_0$. Let $\Delta$ be an $r$-minor ($r \ge 1$) of $\bfX_{[p,q]}$ that does not involve the $p_0$th row and the $q_0$th column. Denote $\Delta'$ the corresponding $r$-minor in $\bfX'_{[p,q]}$ (replace $x_{p,q}$ by $x_{p,q}'$). Then direct computation shows $$\Delta - \Delta' \in  \langle x_{p,q} \mid (p,q) \in \Gamma, p<p_0 \text{ or } q<q_0 \rangle.$$
By (a), $\langle x_{p,q} \mid (p,q) \in \Gamma, p<p_0 \text{ or } q<q_0 \rangle \subseteq I$, so we see that $\Delta-\Delta' \in I \cap I'$.
\item[(c)] Let $\Delta$ be any $r$-minor in $\bfX$ that involves $\Gamma$ but does not involve $c$. Then $\Delta = \Delta'$ where $\Delta'$ is obtained from $\Delta$ by replacing $x_{p,q} ((p,q) \notin \Gamma)$ by $x_{p,q}'$.
\item[(d)]  Let $\Delta$ be any $r$-minor in $\bfX$ that does not involve $\Gamma$ and let
$\Delta'$ the corresponding $r$-minor in $\bfX'$ (replace $x_{p,q}$ by $x_{p,q}'$).
Denote $\widetilde{\Delta}$ and $\widetilde{\Delta'}$ the $r+1$-minors obtained by adding the $p_0$th row and $q_0$th column to $\Delta$ and $\Delta'$ respectively. Then 
\[c\Delta = \widetilde{\Delta'} \text{ and } c\Delta'=\widetilde{\Delta}.\]
\end{itemize}
Now, we are ready to prove \eqref{I=I'}, $I=I'$.\\
We first show that $I \subseteq I'$.
Fix $(p,q) \in D_{=0}(w) \cup \E_{>0}(w)$. We need to show that the following set
\[\{ x_{p,q} \mid (p,q) \in D_{=0}(w) \} \bigcup \left( \bigcup_{(p,q) \in \E_{>0}(w)}  \{  (r_{p,q}(w)+1)\text{-minors in } \bfX_{[p,q]}
\} \right)
\] is contained in $I'$.
Consider the following cases.
\begin{itemize}
\item[(i)] $(p,q) \in D_{=0}(w)$. We must have $p <p_0$ or $q <q_0$. 
\begin{itemize}
\item[(i.1)]If $(p,q) \in \Gamma$, then $x_{p,q} \in I'$ by construction. 
\item[(i.2)]If $(p,q) \notin \Gamma$, then $(p,q) \in D_{=0}(w')$ and hence $x_{p,q}' \in I_{w'} \subseteq I'$. Therefore, $x_{p,q}= x_{p,q}' + c^{-1}x_{p_0,q}x_{p,q_0} \in I'$, since either $x_{p_0,q}$ or $x_{p,q_0}$ is in $I'$.
\end{itemize}
\item[(ii)] $(p,q) \in \E_{>0}(w)$, say $r_{p,q}(w)=r$. Let $\Delta$ be any $r+1$-minor in $\bfX_{[p,q]}$. In this case, $p >p_0$ by \eqref{D=0}.
\begin{itemize}
\item[(ii.1)] $q <q_0$. In this case, $(p,q) \in \E_{=r}(w')$. If $\Delta$ involves the $p_0$th row, expanding along this row we see that 
$\Delta \in \langle x_{p_0,q} \mid q < q_0 \rangle \subseteq I'$. Otherwise, let $\Delta'$ be the corresponding $(r+1)$-minor in $\bfX_{[p,q]}'$ and we have $\Delta -\Delta' \in \langle x_{p_0,q} \mid q < q_0 \rangle \subseteq I'$ by (b). But $(p,q) \in \E_{=r}(w')$ implies $\Delta' \in I_{w'} \subseteq I'$. So $\Delta \in I'$.
\item[(ii.2)] $q>q_0$. In this case, $(p,q) \in \E_{=r-1}(w')$.
\begin{itemize}
\item[(ii.2.1)] If $\Delta$ involves $c$, then $\Delta = c \left( \Delta \setminus [p_0\mid q_0] \right)'$ by (d), where $\Delta \setminus [p_0\mid q_0] $ is the $r$-minor obtained from $\Delta$ by deleting the row and the column involving $c$. But $(p,q) \in \E_{=r-1}(w')$ implies $\left( \Delta \setminus [p_0\mid q_0] \right)' \in I_{w'} \subseteq I'$. So $\Delta \in I'$.
\item[(ii.2.2)]
If $\Delta$ involves $\Gamma$ but does not involve $c$, then by (c) $\Delta = \Delta'$ where $\Delta'$ is obtained from $\Delta$ by replacing $x_{p,q} ((p,q) \notin \Gamma)$ by $x_{p,q}'$. Expanding $\Delta'$ along the row or column involving $\Gamma$, we see that $\Delta' \in \left\langle r\text{-minors in }\bfX'_{[p,q]}\right\rangle$. But $(p,q) \in \E_{=r-1}(w')$ implies
$\left\langle r\text{-minors in }\bfX'_{[p,q]}\right\rangle \subseteq I_{w'} \subseteq I'$. So $\Delta=\Delta' \in I'$.
\item[(ii.2.3)]
If $\Delta$ does not involve $\Gamma$, then $c\Delta = \widetilde{\Delta'}$ by (d). Expanding the $(r+2)$-minor $\widetilde{\Delta'}$ along the row and the column involving $\Gamma$, we see that $\widetilde{\Delta'} \in \left\langle r\text{-minors in }\bfX'_{[p,q]} \right\rangle$. Again, 
$\left\langle r\text{-minors in }\bfX'_{[p,q]}\right\rangle  \subseteq I'$ since $(p,q) \in \E_{=r-1}(w')$. So $\Delta = c^{-1} \widetilde{\Delta'} \in I'$.
\end{itemize}
\end{itemize}
\end{itemize}

Conversely, we show that $I' \subseteq I$.
Fix $(p,q) \in D_{=0}(w') \cup \E_{>0}(w')$. Again, we show that the set
\[\{ x_{p,q}' \mid (p,q) \in D_{=0}(w') \} \bigcup \left( \bigcup_{(p,q) \in \E_{>0}(w')}  \{  (r_{p,q}(w')+1)\text{-minors in } \bfX_{[p,q]}'
\} \right)
\] is contained in $I$.
\begin{itemize}
\item[(i)] $(p,q) \in D_{=0}(w')$.
\begin{itemize}
\item[(i.1)] If $p<p_0$ or $q<q_0$, then $(p,q) \in D_{=0}(w)$ and hence $x_{p,q} \in I$. By (a), either $x_{p_0,q}$ or $x_{p,q_0}$ is in  $I$. Therefore, $$x'_{p,q} = x_{p,q} - c^{-1}x_{p_0,q}x_{p,q_0} \in I.$$
\item[(i.2)] If $p<p_0$ and $q>q_0$, then $(p,q) \in \E_{=1}(w)$. Hence, the $2$-minor $cx'_{p,q} = cx_{p,q}-x_{p_0,q}x_{p,q_0} \in I_w \subseteq I$. So $x'_{p,q} \in I$.
\end{itemize}
\item[(ii)] $(p,q) \in \E_{>0}(w')$. In this case, $p>p_0$ by \eqref{D=0}. Suppose $r_{p,q}(w')=r$ and let $\Delta'$ be any $r+1$-minor in $\bfX_{[p,q]}'$.
\begin{itemize}
\item[(ii.1)] If $q<q_0$, then $(p,q) \in \E_{=r}(w)$. By (d), $c\Delta' = \widetilde{\Delta}$. Expanding $\widetilde{\Delta}$ along the $q_0$th column, we see  $\widetilde{\Delta} \in \left\langle (r+1)\text{-minors in }\bfX_{[p,q]} \right\rangle$. But $(p,q) \in \E_{=r}(w) $ implies 
$\left\langle (r+1)\text{-minors in }\bfX_{[p,q]} \right\rangle \subseteq I$. So $\Delta' = c^{-1}\widetilde{\Delta} \in I$.
\item[(ii.2)] If $q>q_0$, then $(p,q)  \in \E_{=r+1}(w) $. Again, $c\Delta' = \widetilde{\Delta}$ by (d). This time $\widetilde{\Delta} \in I_w \subseteq I$ since $(p,q)  \in \E_{=r+1}(w) $. Therefore, $\Delta' = c^{-1}\widetilde{\Delta} \in I$.
\end{itemize}
\end{itemize}
\end{proof}
\begin{example}
Consider $w = 35142$ in $S_5$. Use the same notation as in the proof of Theorem \ref{frational}. We have $c=x_{13}$,
$I_w = \langle x_{11},x_{12},x_{21},x_{22}, [12\mid 34],[34\mid 12] \rangle$
and $I_{w'} = \langle x_{21}', x_{22}',x_{24}',[34 \mid 12]'\rangle$.
Check that the following elements lie in $ \langle x_{11},x_{12} \rangle$ and hence in $ I \cap I'$:
 $$x_{21}-x_{21}',x_{22}-x_{22}', [34 \mid 12]- [34 \mid 12]', [12 \mid 34]-c^{-1}x_{24}'.$$
Therefore, we see that $I=I'$ and $R_w[c^{-1}] = R_{w'}[x_{11},x_{12},c^{-1}]$.\\
In the following diagram, the $1$'s indicate the permutation and the dots indicate the elements in $D_{>0}(w)$.

\begin{center}
\begin{footnotesize}\begin{tabular}{  | c | c | c | c | c |}
    \hline
       &    & 1 &    &        \\ \hline
       &    &    & $\bullet$ & 1   \\ \hline
   1  &    &    &    &        \\ \hline
       &$\bullet$  &    &  1&     \\ \hline
     &  1  &    &    &       \\ \hline
      \end{tabular}\end{footnotesize}
  \end{center}

\end{example}

\section{Complete intersection matrix Schubert varieties} \label{ci}
We want to characterize the complete intersection MSVs.
By \eqref{reducetopermutation}, we may assume $w \in S_n$. Denote $w_{(p,q)}$ the $r_{p,q}(w) \times r_{p,q}(w)$ submatrix of $w$ involving the rows
of indices $p-r_{p,q}(w), \dots, p-1$ and the columns of indices $q-r_{p,q}(w), \dots, q-1$.
Define the submatrix $\bfX_{(p,q)}$ of $\bfX$ similarily. Furthermore, denote $\bfX_{\overline{(p,q)}}$ the $(r_{p,q}(w)+1) \times (r_{p,q}(w)+1)$ submatrix of $\bfX$ involving the rows of indices
$p-r_{p,q}(w), \dots, p$ and the columns of indices $q-r_{p,q}(w), \dots, q$.
If $r_{p,q}(w)=0$, $\bfX_{\overline{(p,q)}} = \det \bfX_{(p,q)} = x_{p,q}$. However, in order to make our proof more transparent, we will only use $\bfX_{\overline{(p,q)}}$ for $(p,q) \in D_{>0}(w)$.

Recall that the codimension of $\X_w$ in $M_{n \times n}$ is $|D(w)|$.
So $\X_w$ is a complete intersection if and only if $I_w$ can be generated by $|D(w)|$ many elements.
We need the following lemma.
\begin{lemma} \label{diagram=essential}
Let $w$ be such that $\X_w$ is a complete intersection. Then 
for any $(p,q) \in D_{>0}(w)$ and any $1 \le i \le r_{p,q}(w)$,
\[ (p-i,q) \notin D(w) \text{ and } (p,q-i) \notin D(w).\] In particular,
\[D_{>0}(w) \subseteq \E(w),\text{ \ie} D_{>0}(w) = \E_{>0}(w).\]
\end{lemma}
\begin{proof}
We only have to prove the first statement. The second statement follows by definition. 
Suppose the first statement does not hold, then by symmetry we may assume that there exist $(p_0,q_0)$ in $D_{>0}(w)$ and some $1 \le i_0 \le  r_{p_0,q_0}(w) $ so that $(p_0-i_0,q_0) \in D(w)$ and $(p_0-j,q_0)\notin D(w)$ for all $1 \le j <i_0$. 
\\
Denote $r := r_{p_0,q_0}(w) $. 
Consider the $(r+1)$-minor $$\Delta=[p_0-r-1, \dots, \widehat{p_0-i_0}, \dots, p_0 \mid q_0-r, \dots, q_0]$$ 
in the $(r+2)\times(r+1)$ submatrix $\bfX_{\Delta}$ of $\bfX$ involving the rows of indices $p_0-r-1,\dots,p_0$ and the columns of indices $q_0-r,\dots,q_0$. \\
Consider
 the set \[G =  \{ x_{p,q} \mid (p,q) \in D_{=0}(w) \} \bigsqcup \{ \det \bfX_{\overline{(p,q)}} \mid (p,q) \in D_{>0}(w) \} \bigsqcup \{\Delta \}.\]
Observe that the unions are disjoint and $|G| = |D(w)|+1$.
Denote $\bm$ the maximal graded ideal of $K[\bfX]$.
We claim that 
\begin{equation}\label{indep}
\text{ the image of } G \text{ in } I_w/\bm I_w \text{ form a } K(=K[\bfX]/\bm) \text{-linearly independent set.}
\end{equation}
 By Nakayama's lemma, $I_w$ is generated by at least $|D(w)|+1$ elements. This contradicts  the assumption that $\X_w$ is a complete intersection.
\\
It remains to show the claim \eqref{indep}.
Suppose \[
\left( \sum_{(p,q) \in D_{=0}(w)} c_{p,q} x_{p,q} \right) + 
\left( \sum_{(p,q) \in D_{>0}(w)} c_{p,q} \det \bfX_{\overline{(p,q)}} \right)+ c_{\Delta} \Delta  \in \bm I_w
\] where $c_{p,q}$ and $c_{\Delta}$ are in $K$.
Notice $\bm I_w$ is a homogeneous ideal whose generators are of degree at least $2$. This implies $c_{p,q}=0$ for all $(p,q) \in D_{=0}(w)$ since otherwise we will have an element in $\bm I_w$ that has a nonzero degree $1$ part. Therefore, 
\[ F:= \left( \sum_{(p,q) \in D_{>0}(w)} c_{p,q} \det \bfX_{\overline{(p,q)}} \right)+ c_{\Delta} \Delta \in \bm I_w
\]
Fix any antidiagonal term order on $K[\bfX]$.
Then $\ini(F)$ is the antidiagonal term of one of the minors in 
\[ \{ \det \bfX_{\overline{(p,q)}} \mid (p,q) \in D_{>0}(w) \} \bigsqcup \{\Delta \}.
\] So $\ini(F)$ is in the generating set of $\ini(I_w)$ described in
Theorem~\ref{initial}{\it (1)}.
On the other hand, $F \in \bm I_w$ implies that there exists an antidiagonal $\delta$ in the generating set of $\ini(I_w)$ described in Theorem~\ref{initial}{\it (1)} such that $\delta$ is a factor of $\ini(F)$ but $\delta \neq \ini(F)$.
This means that $\delta \in \ini(I_w)$ is an antidiagonal in one of the submatrices of the matrices in
\[
\{ \bfX_{\overline{(p,q)}} \mid (p,q) \in D_{>0}(w) \} \bigsqcup \{ \bfX_{\Delta}\},
\] and that $\delta$ is  of size $\leq r_{p,q}(w)$ (if $\delta$ is in $\bfX_{\overline{(p,q)}}$) or of size $\leq r_{p_0,q_0}(w)$ (if $\delta$ is in $\bfX_{\Delta}$). This is impossible in view of Theorem \ref{initial}{\it (1)}.
Therefore, we conclude that $c_{p,q}(w) = c_{\Delta} = 0$ for all $(p,q) \in D(w)$ and the claim \eqref{indep} is proved.
\end{proof}
Now we are ready for 
\begin{theorem}\label{thmCI} Let $w \in S_n$. The following conditions are equivalent:
\begin{enumerate}
\item $\X_w$ is a complete intersection,
\item for any $(p,q) \in D_{>0}(w)$, $w_{(p,q)}$ is a permutation in $S_{r_{p,q}(w)}$ such that $\X_{w_{(p,q)}}$ is a complete intersection.
\end{enumerate}
In this case, 
\[ \{ x_{p,q} \mid (p,q) \in D_{=0}(w) \} \bigsqcup \{ \det \bfX_{\overline{(p,q)}} \mid (p,q) \in D_{>0}(w) \}
\] is a set of generators for $I_w$.
\end{theorem}
\begin{proof}
The conditions in {\it (2)} shows that for any $(p,q) \in D_{>0}(w)$ the only nonzero entries of $w_{[p,q]}$ appear in $w_{(p,q)}$, so all size $(r_{p,q}(w)+1)$ minors except $\det \bfX_{\overline{(p,q)}}$ belong to $\langle
 x_{p,q} \mid (p,q) \in D_{=0}(w) \rangle$. 
Therefore,  the last statement follows immediately from \eqref{generator} and the equivalence of {\it (1)} and {\it (2)}.

We prove {\it(2)} implies {\it(1)}.
Let $(p_1,q_1),\dots,(p_t,q_t)$ be all the elements in $D_{>0}(w)$ satisfying
\[ \{ (p,q) \mid p \geq p_i, q \geq q_i \} \cap D_{>0}(w) = \emptyset.
\]
Denote $r_i = r_{p_i,q_i}(w)$ and $w_i = w_{(p_i,q_i)} \in S_{r_i}$.
By the assumptions in {\it(2)}, the diagram $D(w_i)$ of $w_i$ is contained in $D(w)$. Also, the ideal $I_{w_i}$ of $\X_{w_i}$ is generated by $|D(w_i)|$ many elements, since $\X_{w_i}$ is a complete intersection.
Note also that the conditions in {\it (2)} and the choice of $(p_i,q_i)$ imply that $D(w)$ can be decomposed as
\[
 \left\{ (p,q)  \bigg| (p,q)\in D_{=0}(w) \setminus \bigsqcup^t_{i=1} D(w_i) \right\}  \bigsqcup  \left( \bigsqcup^t_{i=1} D(w_i) \right)
 \bigsqcup \left\{ (p_i,q_i) \mid 1\leq i \leq t \right\} 
\] where the unions are disjoint.
Furthermore, by construction one can check using \eqref{generator}  that
\[
I_w = \left\langle x_{p,q} \bigg| (p,q) \in D_{=0}(w) \setminus \bigsqcup^t_{i=1} D(w_i) \right\rangle + \sum^t_{i=1} I_{w_i} + \sum^t_{i=1} \left\langle \det \bfX_{\overline{(p_i,q_i)}} \right\rangle.
\] So $I_w$ is generated by 
\[ \left| \left\{ (p,q) \in D_{=0}(w) \bigg| (p,q) \notin \bigsqcup^t_{i=1} D(w_i)  \right\} \right| +
\sum^t_{i=1} \left| D(w_i) \right| +t = \left| D(w)\right|
\] many elements. Therefore, $\X_w$ is a complete intersection.
\vskip 10pt
To prove {\it (1)} implies {\it (2)}, let $(p,q) \in D_{>0}(w)$ and use induction on $r_{p,q}(w)$.
When $r_{p,q}(w)=1$, we must have $w_{p-1,q-1}=1$ since otherwise either $(p-1,q) \in D_{=1}(w)$ or $(p,q-1) \in D_{=1}(w)$ which contradicts Lemma \ref{diagram=essential}. Also, $\X_{w_{(p,q)}}$ is the affine line, so we are done for $r_{p,q}(w)=1$.\\
Suppose $r_{p,q}(w) >1$ and denote $r=r_{p,q}(w)$. By Lemma \ref{diagram=essential}, 
\[
(p-i,q) \notin D(w) \text{ and } (p,q-i) \notin D(w), \text{ for }i=1, \dots,r.
\]
This implies that $\rank (w_{(p,q)})=r$ and $w_{(p,q)} \in S_r$. 
Moreover, consider the diagram $D(w_{(p,q)})$. This is exactly the part of the diagram $D(w)$ involving the rows of indices $p-r, \dots, p-1$ and the columns of indices $q-r, \dots, q-1$. For any $(p',q') \in D_{>0}(w_{(p,q)})$, $r_{p',q'}(w_{(p,q)})=r_{p',q'}(w)<r$.
So by inductive hypothesis, $w_{(p,q)} \in S_r$ satisfies the conditions in (2).
Therefore, by the implication of {\it (2)}$\Rightarrow${\it (1)} we just proved, $\X_{w_{(p,q)}}$ is a complete intersection as desired.
\end{proof}
\begin{example} Consider $w \in S_6$.
Again, In the following diagrams the $1$'s indicate the permutation and the dots indicate the elements in $D_{>0}(w)$.
\begin{itemize}
\item[(1)] If $w= 361452$, then $\overline{X}_w$ is not a complete intersection since $D_{>0}(w) = \{ (2,4),(2,5),(4,2),(5,2) \}$ but $(2,5) \notin \E_{>0}(w)$. 
\begin{center}
\begin{footnotesize}\begin{tabular}{ | c | c | c | c | c | c |}
    \hline
       &    & 1 &    &    &    \\ \hline
       &    &    & $\bullet$ & $\bullet$ & 1 \\ \hline
    1 &    &    &    &    &    \\ \hline
       & $\bullet$ &    & 1 &    &    \\ \hline
       & $\bullet$   &    &    & 1  &    \\ \hline
       & 1  &    &    &    &    \\ \hline
  \end{tabular}\end{footnotesize}
  \end{center}
\item[(2)] If $w=352614$, then $D_{>0}(w) = \{(2,4),(4,4)\}=\E_{>0}(w)$. But $\overline{X}_w$ is still not a complete intersection since $w_{(4,4)} \notin S_2$. So we see that the condition $D_{>0}(w) =\E_{>0}(w)$ is not sufficient for $\overline{X}_w$ to be a complete intersection.

\begin{center}
\begin{footnotesize}\begin{tabular}{ | c | c | c | c | c | c |}
    \hline
       &    & 1 &    &    &    \\ \hline
       &    &    & $\bullet$ & 1 &  \\ \hline
     &  1  &    &    &    &    \\ \hline
       &  &    & $\bullet$ &    &   1 \\ \hline
     1  &    &    &    &   &    \\ \hline
       &   &    &  1  &    &    \\ \hline
  \end{tabular}\end{footnotesize}
  \end{center}
\item[(3)] If $w=462153$, then $\overline{X}_w$ is a complete intersection, and
$I_w$ is generated by $\{x_{i,j} \mid 1\le i \le 2, 1 \le j \le 3 \} \bigcup \{x_{31}, \det \bfX_{\overline{(2,5)}}, \det \bfX_{\overline{(5,3)}} \} $

\begin{center}
\begin{footnotesize}\begin{tabular}{ | c | c | c | c | c | c |}
    \hline
       &    &  &  1  &    &    \\ \hline
       &    &    & &$\bullet$ & 1   \\ \hline
     &  1  &    &    &    &    \\ \hline
     1  &  &    &  &    &    \\ \hline
      &    &  $\bullet$  &    &  1 &    \\ \hline
       &   &  1  &    &    &    \\ \hline
  \end{tabular}\end{footnotesize}
  \end{center}
\end{itemize}
\end{example}
\bibliographystyle{plain}
\bibliography{all}
\end{document}